\documentclass[a4paper,12pt]{amsart}
\usepackage{hyperref,fancyhdr,mathrsfs,amsmath,amscd,amsthm,amsfonts,latexsym,amssymb}

\DeclareMathOperator{\trace}{trace}

\DeclareMathOperator{\dif}{d}

\renewcommand{\H}{\mathscr{H}}
\newcommand{\Hc}{\mathcal{H}}
\newcommand{\V}{\mathscr{V}}
\newcommand{\F}{\mathscr{F}}

\newcommand{\tD}{\mathscr{D}}

\def \a{\alpha}

\def \e{\eta}

\def \phi{\varphi}
\def \Phi{\varPhi}
\def \p{\pi}

\def \t{\tau}

\def \R{\mathbb{R}}

\def \C{\mathbb{C}\,}

\def\widecheckg{g^{\hspace*{-2.5pt}\vbox to 5pt{\hbox to
0pt{\LARGE$\check{}$}}}\hspace*{2pt}}

\def\widecheckl{\lambda^{\hspace*{-3.5pt}\vbox to 8pt{\hbox to
0pt{\LARGE$\check{}$}}}\hspace*{2pt}}

\begin{document}

\title{Harmonic morphisms on Heaven spaces}
\author{Paul Baird and Radu~Pantilie\;\dag}
\thanks{\dag\;Gratefully acknowledges that this work was partially supported by a
C.N.C.S.I.S. grant, code no.\ 811}
\email{\href{mailto:Paul.Baird@univ-brest.fr}{Paul.Baird@univ-brest.fr},
       \href{mailto:Radu.Pantilie@imar.ro}{Radu.Pantilie@imar.ro}}
\address{P.~Baird, D\'epartement de Math\'ematiques, Laboratoire C.N.R.S. U.M.R. 6205,
Universit\'e de Bretagne Occidentale, 6, Avenue Victor Le Gorgeu, CS 93837,
29238 Brest Cedex 3, France}
\address{R.~Pantilie, Institutul de Matematic\u a ``Simion Stoilow'' al Academiei Rom\^ane,
C.P. 1-764, 014700, Bucure\c sti, Rom\^ania}

\newtheorem{thm}{Theorem}[section]
\newtheorem{lem}[thm]{Lemma}
\newtheorem{cor}[thm]{Corollary}
\newtheorem{prop}[thm]{Proposition}

\theoremstyle{definition}

\newtheorem{defn}[thm]{Definition}
\newtheorem{rem}[thm]{Remark}
\newtheorem{exm}[thm]{Example}

\numberwithin{equation}{section}

\maketitle
\thispagestyle{empty}

\section*{Abstract}
\begin{quote}
{\footnotesize  We prove that any (real or complex) analytic horizontally conformal submersion
from a three-dimensional conformal manifold $M^3$ to a two-dimensional conformal
manifold $N^2$ can be, locally, `extended' to a unique harmonic morphism from the
$\Hc$(eaven)-space of $M^3$ to $N^2$.}
\end{quote}

\section*{Introduction}

\indent
\emph{Harmonic morphisms} are maps between Riemannian manifolds (or, more generally, Weyl spaces) which
pull-back (local) harmonic functions to harmonic functions. By a basic result, harmonic morphisms
are characterised as harmonic maps which are horizontally weakly conformal (see \cite{BaiWoo2}\,, \cite{LouPan}\,).\\
\indent
In \cite{Woo-4d} (see \cite{LouPan}\,) it is proved that any (submersive) harmonic morphism $\phi$ with two-dimensional
fibres from a four-dimensional Einstein manifold $(M^4,g)$ is \emph{twistorial}, possibly with respect to the
opposite orientation on $M^4$. That is, after reversing, if necessary, the orientation on $M^4$,
the unique positive almost Hermitian structure $J$ on $(M^4,g)$ with respect to which $\phi$ is holomorphic
is integrable; we then say that $\phi$ is \emph{positive}. Moreover, $J$ is parallel along the fibres of $\phi$\,.
Thus if, further, $(M^4,g)$ is anti-self-dual then $J$ determines a complex surface $S$
in the twistor space of $(M^4,g)$ which, if $J$ is nowhere K\"ahler, is foliated by the holomorphic distribution
determined by the Levi-Civita connection of $g$\,. It follows that, then, $\phi$ is completely determined
by $S$\,. This gives the twistorial construction of all positive harmonic morphisms with two-dimensional fibres from a
four-dimensional Einstein anti-self-dual manifold with nonzero scalar curvature \cite{Woo-4d}\,.\\
\indent
In \cite{LeB-Hspace}\,, it is proved that any analytic three-dimensional conformal manifold $(N^3,c)$
is the conformal infinity of a unique (up to homotheties) four-dimensional Einstein anti-self-dual manifold
$(M^4,g)$ with nonzero scalar curvature; the fact that $(N^3,c)$ is the conformal infinity of $(M^4,g)$ means that
$N^3\sqcup M^4$ is a manifold with boundary (equal to $N^3$) on which $c$ and $g$ determine a conformal structure
with respect to which $g$ has a pole along $N^3$. The constructed Einstein anti-self-dual manifold $(M^4,g)$
is called the \emph{$\Hc$\!(eaven)-space} of $(M^3,c)$ \cite{LeB-Hspace}\,.\\
\indent
The above mentioned two constructions raise the following questions. Given a positive harmonic morphism
with two-dimensional fibres on the $\Hc$-space of a three-dimensional analytic conformal manifold $(N^3,c)$\,,
can it be extended over $N^3$\,? If so, is the resulting map horizontally conformal?
Conversely, can any analytic horizontally conformal submersion, with
one-dimensional fibres on $(N^3,c)$ be, locally, `extended' to a harmonic morphism on its $\Hc$-space?
\emph{In this paper we answer in the affirmative to all of these questions}, thus generalizing results
of \cite{BaiWoo-shearfree} in which the same conclusions are obtained for four-dimensional constant curvature
Riemannian manifolds.\\
\indent
In Section 1 we review a few facts on harmonic morphisms between Weyl spaces. In Section 2 we present
the basic properties of the space of isotropic geodesics of a (three-dimensional) conformal manifold. Here,
the emphasis is on the induced contact structure \cite{LeB-Hspace}\,, \cite{LeB-nullgeod}\,; also,
our presentation seems to be new and more natural. We end Section 2 by illustrating how
this theory can be used to construct horizontally conformal submersions with one-dimensional
fibres from the Euclidean space. This (Example \ref{exm:PB}\,; cf.\ \cite{BaiWoo-shearfree}\,)
also provides examples for the main result of the paper (Theorem \ref{thm:PB}\,), proved in Section 3\,.
There, we also explain how other examples can be obtained by using a construction of \cite{Cal-sds}\,.

\section{Harmonic morphisms between Weyl spaces}

\indent
This paper deals mainly with (real or complex) analytic manifolds and maps. Nevertheless,
the results presented in this section also hold in the category of smooth manifolds.\\
\indent
A \emph{conformal structure} on a manifold $M$ is a line subbundle $c$ of $\otimes^2T^*M$
which, on open sets of $M$, is generated by Riemannian metrics \cite{LeB-nullgeod}\,;
then any such metric is a \emph{local representative} of $c$\,.\\
\indent
If $c$ is a conformal structure on $M$ then $(M,c)$ is a \emph{conformal manifold}. A \emph{conformal connection}
on $(M,c)$ is a linear connection which preserves the sheaf of sections of $c$\,. A \emph{Weyl connection} is a
torsion free conformal connection. If $D$ is a Weyl connection on $(M,c)$ then $(M,c,D)$ is a \emph{Weyl space}.
For more information on Weyl spaces see \cite{Cal-sds}\,, where although the discussion is placed in the setting
of smooth manifolds it can be easily extended to complex manifolds (see \cite{LouPan}\,).\\
\indent
Whilst in the category of complex manifolds, we shall assume that all maps between conformal manifolds
have nowhere degenerate fibres (see \cite{LouPan} for a few facts on harmonic morphisms with
degenerate fibres).

\begin{defn}[cf.\ \cite{BaiWoo2}\,]
Let $\phi:(M,c_M)\to(N,c_N)$ be a map between conformal manifolds. Then $\phi$ is \emph{horizontally
weakly conformal} if, outside the set where its differential is zero, $\phi$ is a Riemannian submersion
with respect to suitable local representatives of $c_M$ and $c_N$.
\end{defn}

\indent
Next, we recall the definitions of harmonic maps and morphisms (see \cite{LouPan}\,, cf.\ \cite{BaiWoo2}\,).

\begin{defn}
A map $\phi:(M,c_M,D^M)\to(N,c_N,D^N)$ between Weyl spaces is \emph{harmonic} if the trace, with respect
to $c_M$, of the covariant derivative of $\dif\!\phi$ is zero.\\
\indent
A \emph{harmonic morphism} between Weyl spaces is a map which pulls-back (local) harmonic functions
to harmonic functions.
\end{defn}

\indent
The following result is basic for harmonic morphisms (see \cite{LouPan} and its references,
cf.\ \cite{BaiWoo2}\,).

\begin{thm}
A map between Weyl spaces is a harmonic morphism if and only if it is a harmonic map which
is horizontally weakly conformal.
\end{thm}

\indent
See \cite{LouPan}\,,\,\cite{LouPan-II} for more information on harmonic morphisms between Weyl spaces
and \cite{BaiWoo2} for harmonic morphisms in the setting of Riemannian manifolds; also, see
\cite{Gudbib} for an up to date list of papers on harmonic morphisms.

\section{Spaces of isotropic geodesics}

\indent
In this section, with the exception of Example \ref{exm:PB}\,, all the manifolds and maps are
assumed to be complex analytic.\\
\indent
Let $(M^3,c_M)$ be a three-dimensional (complex-)conformal manifold and let $\p:P\to M$ be the
bundle of isotropic directions tangent to $(M^3,c_M)$\,. Then $P$ is a locally trivial bundle
with typical fibre the conic $Q_1=\C\!P^1$\,.\\
\indent
Let $D$ be a Weyl connection, locally defined, on $(M^3,c_M)$\,. For any $p\in P$ let
$\F_p\subseteq T_pP$ be the horizontal lift, with respect to $D$\,, of $p\subseteq T_{\p(p)}M$.
Then $\F$ is (the tangent bundle of) a foliation on $P$; moreover, $\F$ does not depend of $D$\,.\\
\indent
The quadruple $\t=(P,M,\p,\F)$ is a twistorial structure on $M$ \cite{PanWoo-sd} (see \cite{LouPan-II} for
the definition of twistorial structures in the smooth category), its twistor space $Z$ (that is, the
leaf space of $\F$) is the \emph{space of isotropic geodesics} \cite{LeB-Hspace}\,,\,\cite{LeB-nullgeod} of $(M^3,c_M)$\,.
Furthermore, locally, $\t$ is simple (that is, by passing, if necessary, to an open set of $M$
there exists a, necessarily unique, complex structure on $Z$ with respect to
which the projection $\p_Z:P\to Z$ is a submersion with connected fibres and each of its
fibres intersects the fibres of $\p$ at most once); then the fibres of $\p$ are mapped by $\p_Z$
diffeomorphically onto the \emph{skies} (see \cite{LeB-Hspace}\,,\,\cite{LeB-nullgeod}\,) of $(M^3,c_M)$\,.\\
\indent
\emph{From now on, in this section, we shall assume that $\t$ is simple.}\\
\indent
Let $\tD^P$ be the distribution on $P$ such that $\tD^P_p$ is the sum of ${\rm ker}\dif\!\p_p$ and
the horizontal lift, with respect to $D$\,, of $p^{\perp}\subseteq T_{\p(p)}M$, for each $p\in P$.
Then $\tD^P$ does not depend of $D$\,. Moreover, $\tD^P$ is projectable with respect to $\F$.
To prove this, we place the discussion on the bundle $E$ of conformal frames on $(M^3,c_M)$\,.
Let $p_0$ be a fixed isotropic direction on $\C^{\!3}$, endowed with its canonical conformal structure,
and let $G$ be the subgroup of ${\rm CO}(3,\C)$ which preserves $p_0$. Then $P=E/G$ and let
$\F^E$ and $\tD^E$ be the preimages of $\F$ and $\tD^P$, respectively, through the differential
of the projection $E\to P$.\\
\indent
It is clear that $\tD^P$ is projectable with respect to $\F$ if and only if
$\tD^E$ is projectable with respect to $\F^E$. Furthermore, we can define $\F^E$ and $\tD^E$,
alternatively, as follows. Let $\xi_0$ be a nonzero element of $p_0$
and let $B(\xi_0)$ be the standard horizontal vector field corresponding to $\xi_0$\,;
that is, $B(\xi_0)$ is the vector field on $E$ which, at each $u\in E$\,, is the horizontal lift
of $u(\xi_0)\in TM$, with respect to $D$\,. Then, by using the same notation for
elements of $\mathfrak{co}(3,\C)$ and the corresponding fundamental vector fields on $E$\,,
the distribution $\F^E$ is generated by $B(\xi_0)$ and all $A\in\mathfrak{g}$\,,
where $\mathfrak{g}$ is the Lie algebra of $G$\,. Similarly, $\tD^E$ is generated by all $B(\xi)$
with $\xi\in {p_0}^{\perp}$ and all $A\in\mathfrak{co}(3,\C)$\,.\\
\indent
From Cartan's structural equations and \cite[Ch III, Prop 2.3]{KoNo}
it follows quickly that $\tD^E$ is projectable with respect to $\F^E$ and, therefore, $\tD^P$ is projectable
with respect to $\F$. Moreover, \cite[Ch III, Prop 2.3]{KoNo} also implies that $\tD^E$ is nonintegrable;
that is, nowhere integrable.
Hence, also, the induced distribution $\tD=\dif\!\p_Z\bigl(\tD^P\bigr)$ on $Z$ is nonintegrable.
As $\dim Z=3$\,, this is equivalent to the fact that $\tD$ is a contact distribution on $Z$\,.

\begin{rem} \label{rem:contact_skies}
1) The skies are tangent to the contact distribution $\tD$. Moreover, at each point, $\tD$ is generated
by the tangent spaces to the skies.\\
\indent
2) Any surface embedded in $Z$ corresponds to a foliation by isotropic geodesics on some open set of $M^3$.
\end{rem}

\indent
All of the above can be quickly generalised to conformal manifolds of any dimension thus obtaining results
of \cite{LeB-nullgeod}\,. Next, we concentrate on features specific to dimension three.

\begin{prop}[\,\cite{LeB-Hspace}\,,\,\cite{LeB-nullgeod}\,, cf.\ \cite{Hit-complexmfds}\,] \label{prop:nullgeodfolns}
Let $\V$ be a foliation by isotropic geodesics on $(M^3,c_M)$\,. Then $\V^{\perp}$ is integrable.
Furthermore, if we denote by $S_{\V}$ the embedded surface in $Z$ corresponding to $\V$
then the foliations induced by $\V$, on the leaves of $\V^{\perp}$, correspond to the curves determined
by $\tD$ on $S_{\V}$\,.
\end{prop}
\begin{proof}
The fact that $\V^{\perp}$ is integrable follows quickly from the fact that $D$ is torsion-free.\\
\indent
Let $p$ be the section of $P$ corresponding to $\V$. Then $\F$ induces a foliation on $p(M)$
(whose leaves are mapped by $\p$ on the leaves of $\V$). As $p(M)$ is an embedded submanifold of $P$
foliated by the fibres of $\p_Z$\,, we have that $S_{\V}=\p_Z(p(M))$ is an embedded surface in $Z$\,.\\
\indent
The last assertion follows from the fact that $\dif\!p(\V^{\perp})=T(p(M))\cap\tD^P$\,.
\end{proof}

\begin{rem}
In Proposition \ref{prop:nullgeodfolns}\,, the fact that $\V^{\perp}$ is integrable follows also
from the fact that $\dif(\p_Z\circ p)(\V^{\perp})=TS_{\V}\cap\tD$ is one-dimensional
and, hence, integrable.
\end{rem}

\indent
Recall \cite{LeB-Hspace} that, under the identification of any sky with $\C\!P^1$,
its normal bundle in $Z$ is $\mathcal{O}(1)\oplus\mathcal{O}(1)$\,. Thus, by applying \cite{Kod} we obtain
that, locally, the skies are contained by a, locally complete, family of projective lines, parametrised by a
four-dimensional manifold $N^4$ which contains $M^3$ as a hypersurface. Let $H^4=N^4\setminus M^3$.
Furthermore, if we denote by $c_N$ the anti-self-dual conformal structure on $N^4$ with respect to which
$Z$ is the twistor space (see \cite{PanWoo-sd}\,) of $(N^4,c_N)$ then:\\
\indent
\quad$\bullet$ $(M^3,c_M)$ is a totally umbilical hypersurface of $(N^4,c_N)$\,;
furthermore, under the identification of the twistor space of $(N^4,c_N)$ with the space of isotropic
geodesics on $(M^3,c_M)$ any self-dual surface $S$ on $(N^4,c_N)$ corresponds to the isotropic
geodesic $S\cap M$ on $(M^3,c_M)$ \cite{LeB-Hspace}\,.\\
\indent
\quad$\bullet$  $\tD|_H$ determines an Einstein representative $g$ of $c_N|_H$, unique up to
homotheties, with nonzero scalar curvature \cite{War-graviton}.

\begin{defn}[\,\cite{LeB-Hspace}\,]
The Einstein anti-self-dual manifold $(H^4,g)$ is called the \emph{$\Hc$-space} of $(M^3,c_M)$\,.
\end{defn}

\indent
In \cite{LeB-Hspace} it is, also, proved that $g$ has a pole of order two along $M^3$.

\begin{exm}[cf.\ \cite{LeB-Hspace}\,] \label{exm:Hspace_hyperbolic}
By identifying, as usual, $\C^{\!3}$ with a subspace of $\C^{\!4}$ we obtain that the space of
isotropic lines (equivalently, geodesics) on $\C^{\!3}$ is equal to the space of self-dual planes
on $\C^{\!4}$ (any self-dual plane of $\C^{\!4}$ intersects $\C^{\!3}$ along an isotropic line and any
isotropic line of $\C^{\!3}$ is obtained this way). Therefore the space $Z$ of isotropic lines
of $\C^{\!3}$ is $\C\!P^3\setminus\C\!P^1$ (this can be also proved directly).\\
\indent
{}From Remark \ref{rem:contact_skies}(1)\,, it follows that the contact structure on $Z$ is induced
by the one-form $\theta=z_1\dif\!z_3-z_3\dif\!z_1-z_2\dif\!z_4+z_4\dif\!z_2$\,,
where $(z_1,\ldots,z_4)$ are homogeneous coordinates on $\C\!P^3$.\\
\indent
On the other hand, on $\C^{\!4}\setminus\C^{\!3}$, the contact form $\theta$ induces, up to homotheties,
the well-known metric of constant curvature
$g=\tfrac{1}{{x_4}^2}\bigl({\dif\!x_1}^2+\cdots+{\dif\!x_4}^2\bigr)$\,.\\
\indent
Thus $(\C^{\!4}\setminus\C^{\!3},g)$ is the $\Hc$-space of $\C^{\!3}$.
\end{exm}

\indent
We end this section by illustrating how Remark \ref{rem:contact_skies}(2) and
Example \ref{exm:Hspace_hyperbolic} can be used to construct horizontally conformal submersions
with one-dimensional fibres on the real Euclidean space.

\begin{exm}[cf.\ \cite{BaiWoo-shearfree}] \label{exm:PB}
With the same notations as in Example \ref{exm:Hspace_hyperbolic}\,, let $S$
be a complex surface in $Z$ given, in homogeneous coordinates, by $z_j=z_j(u,v)$\,, $j=1,\ldots,4$\,.
Suppose that the foliation on $S$ given by $u={\rm constant}$ is tangent to the contact distribution;
equivalently, suppose that the following relation holds
$$z_1\,\frac{\partial z_3}{\partial v}-z_3\,\frac{\partial z_1}{\partial v}=
z_2\,\frac{\partial z_4}{\partial v}-z_4\,\frac{\partial z_2}{\partial v}\;.$$
\indent
Then the map $x=x(u,v)$\,, into $\R^4$, given by $x=(z_1+z_2\,{\rm j})^{-1}(z_3+z_4\,{\rm j})$
is (locally) a real analytic diffeomorphism. Thus, under this diffeomorphism, the projection $(u,v)\mapsto u$
corresponds to a submersion $\widetilde{\phi}$ (locally defined) on $\R^4$. We claim that
$\phi=\widetilde{\phi}|_{\{x_4=0\}}$ is a horizontally conformal submersion on $\R^3$.\\
\indent
Indeed, any horizontally conformal submersion with one-dimensional fibres on $\R^3$ determines, by complexification,
a complex analytic horizontally conformal submersion with one-dimensional fibres on $\C^{\!3}$. Now, any such
submersion on $\C^{\!3}$ (more generally, on any three-dimensional complex conformal manifold) is determined
by the two foliations by isotropic geodesics which are orthogonal to its fibres (see \cite{PanWoo-sd}\,).
In our case, under the correspondence of Remark \ref{rem:contact_skies}(2)\,, these two foliations are determined
by $S$ and its conjugate $\overline{S}$, given by $w_j=w_j(u,v)$\,, $j=1,\ldots,4$\,, where
$w_1=-\overline{z_2}$\,, $w_2=\overline{z_1}$\,, $w_3=-\overline{z_4}$\,, $w_4=\overline{z_3}$\,.\\
\indent
Note that, from Theorem \ref{thm:PB}\,, below, it will follow that $\widetilde{\phi}|_{\{x_4>0\}}$
is a harmonic morphism from the four-dimensional real hyperbolic space.
\end{exm}

\section{Harmonic morphisms on $\Hc$-spaces}

\indent
In this section all the manifolds and maps are assumed to be complex analytic; by a \emph{real manifold/map}
we mean a manifold/map which is the (germ-unique) complexification of a real analytic manifold/map. We shall
further assume that all the even dimensional conformal manifolds are \emph{oriented}; that is, the bundle of conformal
frames admits a reduction to the identity component of the group of conformal transformations.\\
\indent
Next, we recall (\,\cite{PanWoo-sd}\,, see \cite{LouPan}\,) two examples of \emph{twistorial maps}.
See \cite{LouPan-II} and \cite{Pan-tm} for more information on twistorial maps.

\begin{exm}
1) Let $\phi:(M^4,c_M)\to(N^2,c_N)$ be a horizontally conformal submersion between conformal manifolds of dimensions
four and two. Let $J^N$ be the positive Hermitian structure on $(N^2,c_N)$ and let $J^M$ be the (unique)
almost Hermitian structure on $(M^4,c_M)$ with respect to which $\phi:(M^4,J^M)\to(N^2,J^N)$ is holomorphic.
Then $\phi$ is twistorial if $J^M$ is integrable.\\
\indent
2) Let $\phi:(M^4,c_M)\to(N^3,c_N,D^N)$ be a horizontally conformal submersion from a four-dimensional
conformal manifold to a three-dimensional Weyl space. Let $\V={\rm ker}\dif\!\phi$ and let $\H=\V^{\perp}$.
Also, let $D$ be the Weyl connection of $(M^4,c_M,\V)$ (see \cite{LouPan}\,; $D$ is characterized by the
property $\trace_{c_M}(D\V)=0$\,).\\
\indent
\quad{}Let $D_+=D+*_{\H}I^{\H}$ where $I^{\H}$ is the integrability tensor of $H$ (by definition,
$I^{\H}(X,Y)=-\V[X,Y]$\,, for any local horizontal vector fields $X$ and $Y$)
and $*_{\H}$ is the Hodge star-operator of $(\H,c_M|_{\H})$\,.
Then $\phi$ is twistorial if the partial connections on $\H$, over $\H$, induced by $\phi^*(D^N)$ and $D_+$ are equal.
\end{exm}
\indent
Any real harmonic morphism with two-dimensional fibres from an oriented four-dimensional Einstein manifold
(more generally, Einstein--Weyl space) is twistorial, possibly with respect to the opposite orientation on
its domain \cite{Woo-4d} (see \cite{LouPan}\,); we say that the harmonic morphism is \emph{positive} if it
is twistorial with respect to the given orientation on its domain and \emph{negative} otherwise.\\
\indent
Let $(M^3,c_M)$ be a three-dimensional conformal manifold and let $(H^4,g)$ be its $\Hc$-space.
We shall use the same notations as in the previous section.

\begin{thm} \label{thm:PB}
Any horizontally conformal submersion $\phi:(M^3,c_M)\to(Q^2,c_Q)$ can be, locally, extended to a unique twistorial map
$\widetilde{\phi}:(N^4,c_N)\to(Q^2,c_Q)$.\\
\indent
Moreover, $\widetilde{\phi}|_H:(H^4,g)\to(Q^2,c_Q)$ is a harmonic morphism and any positive harmonic morphism
with two-dimensional fibres on $(H^4,g)$ is obtained this way.
\end{thm}
\begin{proof}
Let $Z$ be the space of isotropic geodesics of $(M^3,c_M)$ and let $\H\subseteq TZ$ be the corresponding
contact distribution.\\
\indent
As $\phi$ is horizontally conformal it determines two foliations by horizontal isotropic geodesics on $(M^3,c_M)$\,.
From Proposition \ref{prop:nullgeodfolns}\,, it follows that $\phi$ corresponds to a pair of disjoint surfaces
$S_1\sqcup S_2\subseteq Z$ endowed with the one-dimensional foliations $TS_j\cap\H$, $(j=1,2)$.\\
\indent
But $Z$ is also the twistor space of $(N^4,c_N)$\,. Then $S_1\sqcup S_2$ and the endowed foliations determine
a twistorial map $\widetilde{\phi}:(N^4,c_N)\to(Q^2,c_Q)$ (see \cite{PanWoo-sd}\,). Obviously,
$\widetilde{\phi}|_M=\phi$ and so its uniqueness is a consequence of analyticity.\\
\indent
Conversely, from \cite{Woo-4d} (see \cite{PanWoo-sd}\,), it follows that any positive
harmonic morphism $\e:(H^4,g)\to(Q^2,c_Q)$ determines a pair of foliations by self-dual surfaces
on $(N^4,c_N)$\,. Then the leaves of these two foliations intersect $M^3$ to determine a pair of foliations
$(\V_j)_{j=1,2}$ by isotropic geodesics on $(M^3,c_M)$\,. Let $\phi:(M^3,c_M)\to(Q^2,c_Q)$
be such that ${\rm ker}\dif\!\phi=\bigl(\V_1\oplus\V_2\bigr)^{\perp}$. Then $\phi$ is horizontally conformal
and $\e=\widetilde{\phi}|_H$.
\end{proof}

\begin{rem} \label{rem:PB}
1) If we apply Theorem \ref{thm:PB} to the particular case when $(H^4,g)$ has constant curvature then we
obtain results of \cite{BaiWoo-shearfree}\,.\\
\indent
2) With the same notations as above, if $D^M$ is an Einstein--Weyl connection on $(M^3,c_M)$
with respect to which $\phi:(M^3,D^M)\to(Q^2,c_Q)$ is a harmonic morphism then $\widetilde{\phi}|_H=\phi\circ\psi$
where $\psi:N^4\to M^3$ is the retract \cite{Hit-complexmfds} of $M^3\hookrightarrow N^4$ corresponding to $D^M$.
It follows that $\psi|_H:(H^4,g)\to(M^3,c_M,D^M)$ is a harmonic morphism \cite{LouPan}\,.
\end{rem}

\indent
Note that, $\phi$ and $\widetilde{\phi}$ of Example \ref{exm:PB} satisfy the assertions of Theorem \ref{thm:PB}\,.
Other examples can be built on the following construction.

\begin{exm}[\,\cite{Cal-sds}\,] \label{exm:Cal-sds}
Let $(M^3,c_M,D^M)$ be a three-dimensional Einstein--Weyl space.
Let $h$ be a (local) representative of $c_M$ and let $\a$ be the Lee form of $D^M$ with
respect $h$\,. By passing to an open set of $M^3$, if necessary, let $I$ be an open set of $\C$
containing $0$ such that $t^2S-6\neq0$\,, on $I\times M$, where $S$ is the scalar curvature of $D^M$
with respect to $h$\,.\\
\indent
Let $N=I\times M$ and $H^4=N^4\setminus M^3$, where $M^3$ is embedded in $N^4$ through $M^3\hookrightarrow N^4$,
$x\mapsto(0,x)$\,, ($x\in M^3$). Define the Riemannian metric $g$ on $H^4$ by
\begin{equation*}
g=\frac{1}{t^2}\bigl((1-\tfrac16\,t^2S)h+(1-\tfrac16\,t^2S)^{-1}(\dif\!t+t\a+\tfrac12\,t^2*_M\!\dif\!\a)^2\bigr)\;.
\end{equation*}
\indent
Then $t^2g$ defines an anti-self-dual conformal structure on $N^4$. Moreover,
$(H^4,g)$ is the $\Hc$-space of $(M^3,c_M)$ and the projection $\psi:N^4\to M^3$\,, $(t,x)\mapsto x$\,,
$(\,(t,x)\in N^4)$, is the retract of $M^3\hookrightarrow N^4$ corresponding to $D^M$.\\
\indent
Note that, if $(M^3,c_M,D^M)$ is the Euclidean space and $h$ its canonical metric then we obtain the
$\Hc$-space of Example \ref{exm:Hspace_hyperbolic}\,. Furthermore, the resulting retract is
(the complexification of) a well-known (see, for example, \cite{BaiWoo2}\,) harmonic morphism
of warped-product type.
\end{exm}

\indent
By Remark \ref{rem:PB}(2) and with the same notations as in Example \ref{exm:Cal-sds}\,, for any harmonic morphism
$\phi:(M^3,c_M,D^M)\to(Q^2,c_Q)$ we have that $\widetilde{\phi}=\phi\circ\psi$ is as in Theorem \ref{thm:PB}\,.
Such $\phi$'s can be obtained, for example, by using for $(M^3,c_M,D^M)$ the
R.~S.~Ward and C.~R.~LeBrun construction (see \cite{Cal-sds} and the references therein).

\end{document}